\documentclass[12pt]{article} %,twocolumn
\usepackage{amsmath}
\usepackage{lscape}
\usepackage[section]{placeins}
\usepackage{mathtools}
\usepackage{amsthm,bm}
\usepackage{times}
\usepackage{graphicx}
\usepackage{amssymb}
\usepackage{url} %,hyperref
\usepackage{natbib}
\usepackage{microtype}
\usepackage{authblk}
\usepackage{siunitx}
\usepackage{subcaption}
\usepackage[margin=1in]{geometry}
\usepackage{longtable}
\usepackage{booktabs}
\usepackage{multirow}
\usepackage{xcolor}
\usepackage{comment}
\usepackage{color,soul} 
\usepackage{graphicx} 
\usepackage{caption}
\usepackage{algorithm}
\usepackage{algorithmic}
\usepackage{booktabs} 
\usepackage{float}
\usepackage{minitoc}

\usepackage[toc,page,header]{appendix}
\usepackage{etoc}
\definecolor{darkgreen}{RGB}{0,90,0.1}
\usepackage[colorlinks=true, citecolor=darkgreen, urlcolor=darkgreen]{hyperref}
\usepackage{multicol}
\usepackage{dcolumn}
\usepackage{enumitem}
\usepackage[capitalize,noabbrev]{cleveref}
\usepackage[version=4]{mhchem} 
\usepackage[textsize=tiny]{todonotes}
\usepackage{natbib}
\usepackage{tikz}
\usetikzlibrary{calc}
\usetikzlibrary{arrows.meta,positioning,fit,backgrounds,shapes.geometric,  decorations.pathreplacing}
\hypersetup{
    colorlinks=true,
    linkcolor=blue,
    urlcolor=blue,
    citecolor=red
}

\theoremstyle{plain}

\newtheorem{theorem}{Theorem}[section]
\newtheorem{proposition}[theorem]{Proposition}
\newtheorem{lemma}[theorem]{Lemma}
\newtheorem{corollary}[theorem]{Corollary}
\theoremstyle{definition}

\theoremstyle{remark}
\newtheorem{remark}[theorem]{Remark}

\title{A Range-Deflated Stochastic Lanczos Quadrature Method for Large-Scale Log-Determinant Estimation } 

\vspace{0.5cm}
\author[1,3]{Verlon Roel Mbingui \thanks{Corresponding Author}}
\author[2,4]{Antoine Tambue }
\author[3]{Issa Karambal}

\affil[1]{Department of Mathematics and Statistics, The University of Dodoma, P. O. Box 338, Dodoma, Tanzania}
\affil[2]{Department of Computing Mathematics and Physics, Western Norway University of Applied Sciences, Inndalsveien 28, 5063 Bergen, Norway}
\affil[3]{African Institute for Mathematical Sciences, Research and Innovation Centre, Kigali, Rwanda}
\affil[4]{Department of Mathematics and Applied Mathematics, University of Cape Town, 7701 Rondebosch, South Africa}

\begin{document}
\maketitle

\begin{abstract}
Estimating the logarithm of the determinant of large sparse symmetric positive definite matrices is an important problem in numerical linear algebra, machine learning, Gaussian processes, and uncertainty quantification. We propose a range-deflated stochastic Lanczos quadrature method, termed DSLQ, for matrix-free log-determinant estimation inspired by Hutch ++. The method constructs a randomized approximation space directly from products with a scaled version of the original matrix, and uses the associated orthogonal projector to decompose the trace of matrix logarithm into projected and complementary contributions. Both contributions are evaluated through the Gauss-Lanczos quadrature, avoiding explicit formation of the matrix logarithm or a low-rank matrix-function surrogate. Since the approximation space is generated from the  matrix rather than from the matrix logarithm, the standard Hutch++ approximation bound does not apply directly. We therefore derive an error analysis specific to the DSLQ construction, quantifying the residual induced by the matrix-generated subspace and propagating this effect through the stochastic residual estimator and the Lanczos quadrature approximations. 

Extensive numerical experiments on both  large sparse synthetic matrices ( Gaussian Markov Random fields, Bayesian inverse problem.) and large-scale real-world matrices, demonstrate that our method substantially reduces computational time while retaining competitive accuracy, thereby confirming its effectiveness, scalability, and favorable accuracy-cost trade-off for large-scale log-determinant estimation.
\end{abstract}
\textbf{keywords : }
Log-determinant,  Trace estimation, Lanczos quadrature,  Hutch ++, Large matrices.

\section{Introduction}
\label{sec:introduction}
Scaling machine learning algorithms to very large datasets and models has been a central concern of the research community, and progress has come from advances such as first-order stochastic optimization and randomized linear algebra. A core  operation in this setting is the computation the log-determinant computation of large positive definite (SPD) matrices. It arises, for instance  multivariate Gaussian models, where log-determinant of covariance or precision matrices  serves as a normalization constant and  plays a central role in inference, model selection, and parameter estimation for Gaussian graphical models and Gaussian processes \cite{aune2014parameter,RueHeld2005, latz2025deep}. Log-determinants also appear in Bayesian machine learning applications including sampling and variational inference\cite{aune2013iterative,MacKay2003}, in kernel learning \cite{han2015large}, and in  discrete probabilistic models such as Markov random fields \cite{wainwright2006log, RueHeld2005}.

\vspace{0.5em}
For a symmetric positive definite matrix $Q$, the log-determinant can be expressed as the trace of its matrix logarithm. This connection allows large-scale log-determinant estimation to use methods that approximate the action $\log(Q)v$ without forming $\log(Q)$ explicitly. Krylov subspace methods compute such matrix-function actions using matrix-vector products with $Q$ and evaluations on a much smaller projected matrix \cite{saad1992analysis}. Aune et al. \cite{aune2014parameter} instead used a contour-integral approximation to the logarithm, which requires solving shifted linear systems and choosing a contour that avoids the spectrum of $Q$. Lanczos quadrature provides another route: it approximates quadratic forms involving matrix functions \cite{meurant2006lanczos,golub1994matrices}. Stochastic Lanczos quadrature (SLQ) combines these approximations with randomized trace estimation to estimate quantities such as the log-determinant \cite{ubaru2017fast}. Mbingui et al. \cite{mbingui2026novel} proposed a complementary approach that approximates the action of the matrix logarithm using Léja-point interpolation and then combines it with Hutch++ for log-determinant estimation. The  Léja point method avoids the full orthogonalization of existing Krylov subspace methods for Log-determinant estimation and is demonstrated to be much faster, but has the same order of convergence as the Chebyshev expansion method \cite{han2015large, han2017approximating}. Reducing this cost of orthogonalization with negligible loss of accuracy is a primary
open issue in large-scale matrix-function computation. These methods support applications including Gaussian-process learning, Bayesian evidence computation, and covariance-matrix regularization \cite{dong2017scalable}. 
\vspace{0.5em}\\
Stochastic Lanczos quadrature (SLQ)~\cite{ubaru2017fast} estimates
$\mathrm{tr}(f(Q))$ by combining Hutchinson's estimator with Gauss-Lanczos
quadrature for each quadratic form, and requires
$\mathcal{O}(1/\varepsilon^2)$ quadratic forms for relative accuracy $\varepsilon$. Hutch++~\cite{meyer2021hutch++} reduces this to $\mathcal{O}(1/\varepsilon)$ for positive semidefinite matrices by deflating a randomized low-rank approximation before applying Hutchinson's estimator to the residual; adaptive and exchangeable refinements include Nystr\"om++ and A-Hutch++~\cite{persson2022improved} and XTrace~\cite{epperly2024xtrace}. When the target is a matrix
function, however, every product with $f(Q)$ must itself be
approximated, and the way the deflation space is built determines the cost. Han et al.~\cite{han2023suboptimal} combine Hutch++ with SLQ for $\log\det(Q)$, building the deflation space from $\log(Q)S$ and deriving explicit parameter bounds via projection-cost-preserving sketches; forming $\log(Q)S$ requires a Krylov approximation for each sketch column. Chen and Hallman~\cite{chen2023krylov} observe that such
black-box treatment of $f(Q)$ wastes information, and construct the
low-rank approximation directly from a block Krylov subspace of $Q$, reusing it for the projected trace. Earlier, Saibaba et
al.~\cite{saibaba2017randomized} and Li and Zhu~\cite{li2021randomized} approximated
$\log\det$ through compressions of $Q$ onto randomized subspace-iteration and block Krylov spaces, without a stochastic residual correction, and Persson and Kressner~\cite{persson2023randomized} showed that a Nystr\"om approximation of $Q$ yields provably good low-rank approximations of operator monotone functions $f(Q)$, including $\log(I+Q)$, with a Hutch++-type correction in funNystr\"om++. Block-orthonormal Lanczos
(BOLT)~\cite{yeon2025bolt} instead embeds variance reduction within the
Lanczos process itself. For indefinite matrices and determinants,
Cortinovis and Kressner~\cite{cortinovis2022randomized} analyze
Hutchinson-type estimators without positivity assumptions and correct an error in the quadrature bound of~\cite{ubaru2017fast}.
\vspace{0.5em}\\
In this paper, we developed a matrix-free approach for estimating the log-determinant of large symmetric positive definite matrices using only matrix-vector products inspired by Hutch ++. To this end, we introduced the Range-Deflated Stochastic Lanczos quadrature method (DSLQ), a projector based stochastic Lanczos quadrature method that combines a randomized approximation space with a Hutch++ type residual correction. The randomized subspace is constructed directly from products with the matrix, while the matrix logarithm is accessed implicitly through Gauss-Lanczos quadrature. \vspace{0.5em}\\
The method proposed here, DSLQ, seats between~\cite{han2023suboptimal}
and~\cite{chen2023krylov}. Like~\cite{han2023suboptimal}, it keeps the classical SLQ structure: a deterministic projected part and a stochastic residual part, each evaluated by Gauss-Lanczos quadrature. Unlike \cite{han2023suboptimal}, the deflation space is built from a single block product with a (shifted) scaled matrix, $Y=(J-\gamma I)$, where $\gamma \in \{0,1\}$, $J$is defined  by \Cref{scaled_sigma_matrix_slq} and $S$ represents the Rademacher or Gaussian random matrix defined by \Cref{rad_gau_matrix}, which costs $s$ matrix-vector products instead of roughly $s\times m$-Lancsoz steps.
Unlike~\cite{chen2023krylov} and~\cite{persson2023randomized}, DSLQ does not form a low-rank surrogate of $\log(J)$; the subspace is used only as a deflation space, which keeps the implementation to a QR factorization and independent single-vector Lanczos runs\cite{ubaru2017fast}.\vspace{0.5em}\\
Accordingly, the contribution of DSLQ is not the randomized low-rank/stochastic decomposition itself, but its specific realization as a projector-based, range deflated stochastic Lanczos quadrature estimator. The main contributions of this work are summarized as follows:
\begin{enumerate}
\item[(i)]
We introduce DSLQ, a range-deflated stochastic Lanczos quadrature method for large-scale log-determinant estimation. The main idea is to construct the randomized deflation space from a shifted version of the scaled matrix rather than from matrix-function products. This substantially reduces the cost of the range-construction stage while maintaining a simple matrix-free implementation based on block matrix products, QR factorization, and independent Lanczos runs.
\item[(ii)]
We develop a residual error analysis that is tailored to the shifted range construction. The analysis quantifies how well the randomized subspace captures the dominant components of the matrix logarithm and separates the unavoidable spectral-tail error from the additional error introduced by building the range from the shifted matrix instead of directly from the matrix logarithm.

\item[(iii)]
We establish a theoretical connection between the shifted matrix sketch and the direct sketching of the matrix logarithm. In particular, we show that, under the stated assumptions, the shifted construction can achieve the same expected randomized-range approximation bound as a subspace built directly from the matrix logarithm. We also derive an end-to-end error decomposition that separates the range-approximation error, the stochastic residual-estimation error, and the Gauss-Lanczos quadrature error.

\item[(iv)]
We conduct an extensive numerical evaluation of DSLQ against LQ\cite{ubaru2017fast}, So-SLQ\cite{han2023suboptimal}, and
funNyström \cite{persson2023randomized} and Krylov-aware \cite{chen2023krylov} methods under matched matrix-vector product budgets. The experiments compare alternative range constructions, validate the proposed residual bounds across several spectral regimes, and assess accuracy, computational cost, robustness, and scalability on large sparse benchmark matrices and representative applications in Gaussian random-field likelihood estimation and Bayesian inverse problems.
\end{enumerate}
The rest of the paper is organized as follows.    \Cref{sec_So_slq} describes  the DSLQ method for log-determinant estimation. \Cref{sec:DSLQ_corrected_error} Describes  the  error analysis associated to DSLQ. \Cref{Kry: experiments} reports numerical experiments on  real-world sparse matrices, comparing the performance of DSLQ  with existing methods such as SLQ \cite{ubaru2017fast}, So-SLQ \cite{han2023suboptimal}, funNyström\cite{persson2023randomized}, Krylov-aware \cite{chen2023krylov} . 

\section{ Log determinant estimation using Range-Deflated Stochastic Lanczos Quadrature method (DSLQ)}
\label{sec_So_slq}
The Stochastic Lanczos Quadrature, originally introduced by Ubaru et al. \cite{ubaru2017fast} combines three key ingredients, namely, the stochastic trace estimator, Gaussian quadrature, and the Lanczos algorithm to efficiently estimate the spectral sums of large matrices. To reduce the variance of  SLQ, Zongyuan Han et al. \cite{han2023suboptimal} proposed a variance-reduced SLQ estimator for positive semi-define (PSD) matrix, by integrating  SLQ with Hutch++ by constructing an orthonormal basis spanning the  range of $f(Q)S$. In contrast, in this work, we construct the approximation subspace directly  from $QS$ following the original Hutch++ for symmetric positive definite matrix (SPD) in order to yield a simpler and more computationally efficient framework  for variance reduction. In the following, we present our proposed methodology. 
\subsection{Lanczos Quadrature approximation}
Let $Q \in \mathbb{R}^{n\times n}$ be a real symmetric positive definite (SPD) matrix and let $f: (0,\infty) \to \mathbb{R}$ be a real-valued function defined on the spectrum of $Q$, such as $f(x) = \log x$.
A key observation in~\cite{ubaru2017fast} is that for any unit-norm vector $v \in \mathbb{R}^n$, the quadratic form $v^\top f(Q) v$ admits the following \emph{Riemann-Stieltjes integral} representation:
\begin{equation}
I = v^\top f(Q) v = \sum_{i=1}^n f(\lambda_i)\mu_i^2
= \int_{a}^{b} f(t)\, d\mu(t),
\label{eq_stieltjes}
\end{equation}
where the measure $\mu(t)$ is the piecewise constant function defined as 
\begin{equation}
    \mu(t) =   
    \begin{cases}
      0,\quad \text{if} \quad t<a=\lambda_1,\\
      \sum_{j=1}^{i-1} \mu_j^2 \quad \text{if} \quad \lambda_{i-1}\leq t < \lambda_i, i=2, \dots ,n,\\
      \sum_{j=1}^{n} \mu_j^2 \quad \text{if} \quad b=\lambda_n\leq t,
    \end{cases}
\end{equation}
with the eigenvalues $\lambda_i$ ordered non decreasingly.
Equation~\Cref{eq_stieltjes} expresses $v^\top f(A)v$ as a spectral integral of $f$ against the measure $d\mu$.
To approximate~\Cref{eq_stieltjes}, one may run $m$ steps of the Lanczos algorithm beginning from $v$:
$$
\mathcal{K}_m(Q, v) = \mathrm{span}\{v, Qv, Q^2v, \ldots, Q^{m-1}v\},
$$
yielding an orthonormal basis $V_m = [v_1, \dots, v_m]$ and a symmetric tridiagonal matrix
$$
T_m = V_m^\top Q V_m =
\begin{bmatrix}
\alpha_1 & \beta_1 \\
\beta_1 & \alpha_2 & \ddots \\
& \ddots & \ddots & \beta_{m-1} \\
& & \beta_{m-1} & \alpha_m
\end{bmatrix}.
$$
It is well-established~\cite{golub1994matrices,meurant2006lanczos} that the Lanczos tridiagonalization induces an $m$-point Gaussian quadrature rule for \Cref{eq_stieltjes}:
\begin{equation}
v^\top f(Q) v \approx e_1^\top f(T_m) e_1
= \sum_{k=1}^{m} \tau_k f(\theta_k), \quad \tau_k = [e_1^\top u_k]^2
\label{eq:gauss_quadrature}
\end{equation}
where $(\theta_k,u_k), \quad k =1,\dots m$, are eigenpairs of $T_{m}$  by using $v $ as the starting vector  \cite{ubaru2017fast}.

\subsection{Log determinant estimation using Range-Deflated Stochastic Lanczos Quadrature (DSLQ)}
\label{DSLQ}
In what follows, $Q \in \mathbb{R}^{n\times n}$ denotes a symmetric positive definite (SPD) matrix. The logarithm of the determinant of $Q$ is given by 
\begin{equation}
    \log(\det Q) = \mathrm{tr}(\log (Q )),
    \label{Léja1}
\end{equation}
where 
\begin{equation}
 \mathrm{tr}(\log Q) =\sum_{j=1}^{n}e_j^T\log(Q)e_j,
    \label{DSLQ1}   
\end{equation}
with $e_j$ a standard basis vector. 
For a large matrix $Q$, computing the logarithm of the determinant from the right-hand side of the equation \Cref{DSLQ1} becomes computationally expensive as $n$ grows. Instead, we consider approximating it using the Monte Carlo method, i.e., 
\begin{equation}
\label{TE}
    \mathrm{tr}(\log(Q))\approx \dfrac{1}{m_{vec}}\sum_{j=1}^{m_{vec}}v_j^T\log(Q)v_j,
%\label{Léja8}
\end{equation}
where $v_j\in\mathbb{R}^n$ is a vector whose entries are drawn from $\{-1,1\}$ with equal probability, and $m_{vec}<n$.
 This technique is well known as the standard Hutchinson’s stochastic trace estimator \cite{hutchinson1989stochastic}. 
The main limitation of the standard Hutchinson's trace approximation estimator method in \Cref{TE} comes  from  its slow convergence rate and high variance \cite{aune2014parameter,meyer2021hutch++}. In fact, the method requires an $\mathcal{O}(1/\epsilon^2)$ matrix-vector multiplication to provide satisfactory estimates of order 
$\epsilon$. To  address this, Meyer et al. \cite{meyer2021hutch++} introduced optimal stochastic trace estimation (Hutch++). Hutch++ was designed as an improved estimator that achieves significant variance reduction and requires only multiplication queries of matrix-vectors $\mathcal{O}(1/\epsilon)$ to provide satisfactory estimates.

 To reduce variance, we combined the Lanczos quadrature method \cite{ubaru2017fast} with  Hutch++ - type residual correction.  Since Hutch++ is originally designed for positive semidefinite matrices (PSD), to ensure that $\log(Q)$ is PSD we apply the same technique as in \cite{mbingui2026novel} by defining the normalized matrix as : 
\begin{equation}
J\coloneqq
\begin{cases}
    \frac{1}{\alpha}Q,&\textnormal{if}~~\,\lambda_{min}(Q) <1\\
    Q,&\textnormal{otherwise},  
\end{cases}
\label{scaled_sigma_matrix_slq}
\end{equation}
where $\alpha \le \lambda_{min}(Q)$. Hence, $$ \lambda_{min}(J)\ge 1.$$ 
Consequently, the matrix operator $\log(J)$ is PSD, and it follows that 
\begin{equation}
   \log(\det(Q)) = n\log(\alpha)+\mathrm{tr}( \log(J)).
\end{equation}
To compute $\log(\det(Q))$, it suffices to efficiently estimate $\mathrm{tr}(\log(J))$. This scaling is particularly useful because $\log(J)$ is positive semidefinite, wich allows the standard Hutch++-type residual correction concentration theory to be applied to the stochastic trace estimation component. In this paper, the spectral interval of the matrix $Q$ is estimated using Gershgorin Circle theorem \cite{thomas2013numerical}. In case the Gershgorin yields a negative lower bound due to lack of diagonal dominance, we replace it with a small positive threshold to reflect the known positive definiteness of the matrix. We now describe the approach based on Hutch++ - type residual correction.

\subsection{Projected trace estimation with a Hutch++ -type residual correction}  
We decompose $\mathrm{tr}(\log(J))$ into a deterministic low-rank part  and a stochastic residual part. Let 
\begin{equation}
 S \in \mathbb{R}^{n\times s} \qquad {\text{with}~~ s= m_{vec}/3} 
 \label{rad_gau_matrix}
\end{equation}
 be a Rademacher or Gaussian random matrix. We construct
the randomized range directly from the matrix $J$
\begin{equation}
Y = (J-\gamma I) S, 
\label{eq:power_sketch}
\end{equation} where $\mathrm{orth(Y )}$ is the orthonormal matrix corresponding to the QR decomposition. 

Define the orthogonal projector $P = AA^\top$ with $ A = \operatorname{orth}(Y)$,
we decompose the trace as
\begin{equation}
\mathrm{tr}(\log(J))
= \mathrm{tr}(P\log(J)P) + \mathrm{tr}((\mathrm{I}-P)\log(J)(\mathrm{I}-P)).
\label{eq_trace_decomp_slq}
\end{equation}
The first term is evaluated  deterministically (low-rank part), and the second term represents
the residual stochastic contribution \cite{meyer2021hutch++}. We avoid the  cost required in Han et al. \cite{han2023suboptimal} in computing $\operatorname{orth}(\log(J)S)$ by computing $\operatorname{orth}(JS)$ in order to obtain $A$.

\vspace{0.5em}
\paragraph{\textbf{Projected Contribution}}

Since $A^\top A = \mathrm{I}_{s}$, the deterministic low - rank term simplifies to
\begin{equation}
\mathrm{tr}(P\log(J)P)
= \mathrm{tr}(A^\top \log(J) A)
= \sum_{i=1}^{m_{vec}/3} q_i^\top \log(J) q_i,
\label{eq:deterministic_part}
\end{equation}
where $q_i$ is the $i$th column of $A$.
Each quadratic form $q_i^\top \log(J) q_i$ is approximated using Lanczos Quadrature method given by \Cref{eq:gauss_quadrature}.
\paragraph{\textbf{Stochastic Residual Contribution}}
The residual term in \Cref{eq_trace_decomp_slq} can be written as
\begin{equation}
\mathrm{tr}_{\mathrm{res}}
= \mathrm{tr}((\mathrm{I}-P)\log(J)(\mathrm{I}-P))
= \mathbb{E}[v^\top (\mathrm{I}-P)\log(J)(I-P)v].
\label{eq:residual_def}
\end{equation}
We estimate this with $r=m_{vec}/3$ random probes $\{v_{j}\}_{j=1}^{m_{vec}/3}$.
Each probe is projected onto the orthogonal complement of $\mathrm{span}(A)$:
\begin{equation}
z_j = (\mathrm{I} - AA^\top) v_j.
\end{equation}
Averaging over $m_{vec}/3$ probes gives
\begin{equation}
\mathrm{tr}_{res} \approx \widehat{\mathrm{tr}}_{\mathrm{res}}
= \frac{3}{m_{vec}}\sum_{j=1}^{m_{vec}/3} z_j^\top \log(J) z_j,
\end{equation}
where the corresponding quadratic form $z_j^\top \log(J) z_j$ is approximated using \Cref{eq:gauss_quadrature}.
Combining both parts from~\Cref{eq:deterministic_part} and \Cref{eq:residual_def},
the proposed Range-Deflated Stochastic Lanczos Quadrature (DSLQ)
estimator for $\log\det(Q)$ is
\begin{equation}
\mathrm{tr}(\log(J))\approx\widehat{\mathrm{tr}(\log(J)})=\mathrm{tr}(A^\top \log(J)A) + \widehat{\mathrm{tr}}_{res}.
\label{eq_vrslq_final}
\end{equation}
The method is summarized by the following algorithm 
\begin{algorithm}
\caption{Approximation of $\log\det(Q)$ using DSLQ}
\label{algo_log_det_DSLQ}
\begin{algorithmic}[1]
 \REQUIRE SPD matrix, $Q\in \mathbb{R}^{n\times n}$. The number of queries, $m_{vec}$.
 \STATE Estimate Spectral interval, $[\lambda_{min},\lambda_{max}]$ using Gershgorin Circle theorem \cite{thomas2013numerical}.
 \STATE Define a scaled matrix $J\in \mathbb{R}^{n\times n} $ as in \Cref{scaled_sigma_matrix_slq}. 
 \STATE  Sample sketching matrix $S\in \mathbb{R}^{n\times \frac{1}{3}m_{vec}}$ a Rademacher (or Gaussian) random matrix with i.i.d $\{+1,-1\} $ with probability $\dfrac{1}{2}$.
 \STATE  Compute $Y=(J-\gamma I)S $, $\gamma \in \{0,1\}$ . 
 \STATE  Form an orthonormal basis $A$ ($A \in \mathbb{R}^{n\times  \frac{1}{3}m_{vec}}$) for the span of $Y$ via $QR$ decomposition.
\FOR{$i=1,m_{vec}/3$}
\STATE Form a unit vector $q_i=A(:,i)$
\STATE $T $= Lanczos($J,q_i,m$) ( Use $q_i$ as starting vector and apply $m$ Lanczos steps on $J$).
\STATE $[X,\Theta] = \mathrm{eig}(T)$ and compute $\tau_k=[e_1^\top,x_k]^2$ for $k=1,\dots m$
\STATE $\mathrm{tr}(A^\top \log(J)A) \gets \mathrm{tr}(A^\top \log(J)A) + \sum_{k=1}^{m} \tau_k f(\theta_k)  $
\ENDFOR
\FOR{$j=1, m_{vec}/3$}
\STATE Form a unit vector $z_j= \dfrac{(\mathrm{I}-AA^\top){v_j}}{\Vert (\mathrm{I}-AA^\top){v_j} \Vert_2}$ (with $v_j$ Rademacher or gaussian vectors)
\STATE $T'$ = Lanczos($J,z_j,m'$) ( Use $z_j$ as starting vector and apply $m'$ Lanczos steps on $J$).
\STATE  $[X',\Theta'] = \mathrm{eig}(T')$ and compute $\tau'_k=[e_1^\top,x'_k]^2$ for $k=1,\dots m'$
\STATE $\widehat{\mathrm{tr}}_{res} \gets  \widehat{\mathrm{tr}}_{res} + \Vert (\mathrm{I}-AA^\top)v_j\Vert_2^2 \sum_{k=1}^{m'} \tau'_k f(\theta'_k)$
\ENDFOR
\STATE $\widehat{\mathrm{tr}}_{res} \gets \dfrac{3}{m_{vec}}\widehat{\mathrm{tr}}_{res}$
 \STATE  $\widehat{\mathrm{tr}(\log(J))}=\mathrm{tr}(A^\top \log(J)A) + \widehat{\mathrm{tr}}_{res}$
 \STATE \textbf{Output: } $ \widehat{\log(\det(Q))} = n\log(\alpha)+\widehat{\mathrm{tr}(\log(J))} .$
\end{algorithmic}
\end{algorithm}
\newpage

\section{ Range-induced and stochastic error analysis of DSLQ}
\label{sec:DSLQ_corrected_error}
In this section, we establish an error analysis for DSLQ that distinguishes the range-induced approximation error from the stochastic trace-estimation error and the Lanczos quadrature error. 
%\subsection*{Setting}
Let $Q\in\mathbb{R}^{n\times n}$ be SPD and let $0<\alpha\le\lambda_{\min}(Q)$. Set $$J=Q/\alpha ~~, \text{so that}~~ J\succeq I$$ and
$E:=\log(J)\succeq 0$, with $$\log\det(Q)=n\log\alpha+\mathrm{tr}(E).$$ Define 
$J=U\Lambda U^T$ with $\lambda_1\ge\cdots\ge\lambda_n\ge 1$, and
partition
\begin{equation}
\Lambda=\begin{bmatrix}\Lambda_1&0\\0&\Lambda_2\end{bmatrix},\qquad
L_1=\log(\Lambda_1),\quad L_2=\log(\Lambda_2),\qquad
U^TS=\begin{bmatrix}S_1\\S_2\end{bmatrix},
\end{equation}
with $$\Lambda_1,L_1\in\mathbb{R}^{k\times k}, S_1\in\mathbb{R}^{k\times s},
S_2\in\mathbb{R}^{(n-k)\times s}~~ \text{and}~ s=k+p=m_{vec}/3.$$ For a shift
$\gamma\in\{0,1\}$ define the sketch, basis and projector
\begin{equation}\label{eq_sketch}
Y_\gamma=(J-\gamma I)S,\qquad A=\mathrm{orth}(Y_\gamma),\qquad P=AA^T .
\end{equation}
%The choice $\gamma=0$ is the construction of Algorithm~1. The choice $\gamma=1$ has the same cost, since $Y_1=JS-S$. 
Finally let
$$R_P=(I-P)E(I-P),$$ so that $$\mathrm{tr}(E)=\mathrm{tr}(PEP)+\mathrm{tr}(R_P).$$

\begin{lemma}[Frobenius range-induced residual bound]\label{lem_range_induced}
Assume $\lambda_k>\gamma$ and that $S_1$ has full row rank. Then
\begin{equation}\label{eq_lemma_DSLQ}
\Vert R_P\Vert_F \le\Vert (I-P)E\Vert _F,\qquad
\Vert(I-P)E\Vert_F^2 \le \Vert L_2\Vert_F^2
+\bigl\Vert(\Lambda_2-\gamma I)S_2S_1^{\dagger}(\Lambda_1-\gamma I)^{-1}L_1\bigr\Vert_F^2 .
\end{equation}
\end{lemma}

\begin{proof}
The first inequality follows from
$$\Vert(I-P)E(I-P)\Vert_F\le\Vert(I-P)E\Vert_F\,\Vert I-P\Vert_2.$$ For the second,
$$U^TY_\gamma=\bigl[\,(\Lambda_1-\gamma I)S_1;\;(\Lambda_2-\gamma I)S_2\,\bigr].$$
Multiplying on the right by $S_1^{\dagger}(\Lambda_1-\gamma I)^{-1}$ gives
$$\bigl[\,I_k;\;Z_\gamma\,\bigr]$$ with
$$Z_\gamma=(\Lambda_2-\gamma I)S_2S_1^{\dagger}(\Lambda_1-\gamma I)^{-1},$$
so $\mathrm{range}\bigl(U[I_k;Z_\gamma]\bigr)\subseteq\mathrm{range}(A)$. Since $PE$ is the
best Frobenius-norm approximation of $E$ among matrices with columns in
$\mathrm{range}(A)$, we may compare with the candidate
$U[I_k;Z_\gamma]\,L_1U_1^T$, where $U_1$ contains the first $k$ columns
of $U$:
\begin{align*}
   \Vert(I-P)E\Vert_F^2&\le
\Bigl\Vert U\begin{bmatrix}0&0\\-Z_\gamma L_1&L_2\end{bmatrix}U^T\Bigr\Vert_F^2
=\Vert Z_\gamma L_1\Vert_F^2+\Vert L_2\Vert_F^2 \\
&=\Vert (\Lambda_2-\gamma I)S_2S_1^{\dagger}(\Lambda_1-\gamma I)^{-1} L_1\Vert_F^2+\Vert L_2\Vert_F^2
\end{align*}

%this implies 
% $$
% \Vert(I-P)E\Vert_F^2\le
% \Vert Z_\gamma (\Lambda_2-\gamma I)S_2S_1^{\dagger}(\Lambda_1-\gamma I)^{-1} L_1\Vert_F^2+\Vert L_2\Vert_F^2. 
% $$
\end{proof}

\begin{theorem}[Expected residual for a Gaussian sketch]\label{thm:gauss}
Under the assumptions of Lemma~\ref{lem_range_induced}, let $S$ have i.i.d.
standard Gaussian entries and $p\ge 2$. Then
\begin{align}
\mathbb{E}_S\Vert(I-P)E\Vert_F^2
&\le \Vert L_2\Vert_F^2+\frac{\Vert\Lambda_2-\gamma I\Vert_F^2}{p-1}\,
\bigl\Vert(\Lambda_1-\gamma I)^{-1}L_1\bigr\Vert_F^2 \\
&= \label{eq_thm_DSLQ}\sum_{i=k+1}^{n}\log^2\lambda_i
+\frac{1}{p-1}\sum_{i=k+1}^{n}\sum_{j=1}^{k}
\frac{(\lambda_i-\gamma)^2\log^2\lambda_j}{(\lambda_j-\gamma)^2}. 
\end{align}
\end{theorem}

\begin{proof}
By rotational invariance, $S_1$ and $S_2$ are independent Gaussian
matrices. Conditioning on $S_1$,
$$\mathbb{E}_{S_2}\Vert(\Lambda_2-\gamma I)S_2X\Vert_F^2=\Vert\Lambda_2-\gamma I\Vert_F^2\Vert X\Vert_F^2$$
for fixed $$X=S_1^{\dagger}(\Lambda_1-\gamma I)^{-1}L_1.$$
Moreover
$$\Vert X\Vert_F^2=\mathrm{tr}\bigl(C^T(S_1S_1^T)^{-1}C\bigr)$$ with
$$C=(\Lambda_1-\gamma I)^{-1}L_1,$$ and
$$\mathrm{E}(S_1S_1^T)^{-1}=I_k/(p-1)$$ for $p\ge2$. Taking expectations in
Lemma~\ref{lem_range_induced} gives~\eqref{eq_thm_DSLQ}.
\end{proof}

\begin{proposition}[Sketching $J$ is no worse than sketching $\log (J)$]
\label{prop_sketch_J}
Fix $\gamma \in \{0,1\}$ and let $g_\gamma(x)=(x-\gamma)/\log (x)$ for $x>1$. If $g_\gamma$ is non decreasing on the spectrum of $J$, then sketching $J$ gives the same expected Frobenius error bound as sketching $\log(J)$ : 
\begin{equation}\label{eq:hmt}
\mathbb{E}_S\Vert(I-P)E\Vert_F^2\;\le\;\Bigl(1+\frac{k}{p-1}\Bigr)\Vert L_2\Vert_F^2 ,
\end{equation}
which is the bound of~\cite[Thm.~10.5]{halko2011finding} for a sketch formed
from $E=\log(J)$ itself, as used in \cite{han2023suboptimal}.
The monotonicity condition holds
\begin{itemize}
\item for $\gamma=1$ whenever $\alpha\le\lambda_{\min}(Q)$, i.e.\
always in the present setting; and
\item for $\gamma=0$ whenever $\lambda_n\ge e$, e.g.\ after rescaling
with $\alpha\le\lambda_{\min}(Q)/e$.
\end{itemize}
\end{proposition}

\begin{proof}
For $i>k\ge j$ we have $\lambda_i\le\lambda_j$, so
$g_\gamma(\lambda_i)\le g_\gamma(\lambda_j)$, which is equivalent to
$$(\lambda_i-\gamma)^2\log^2\lambda_j/(\lambda_j-\gamma)^2\le\log^2\lambda_i.$$
If $\lambda_i=1$ and $\gamma=1$ the left-hand side vanishes. Summing
over $j\le k$ bounds the double sum in~\eqref{eq_thm_DSLQ} by
$k\Vert L_2\Vert_F^2$. For monotonicity,
$$g_1'(x)=(\log x-1+1/x)/\log^2x\ge0$$ because $\log x\ge1-1/x$, and
$$g_0'(x)=(\log x-1)/\log^2x\ge0$$ for $x\ge e$.
\end{proof}

\begin{lemma}[Gauss-Lanczos quadrature error]\label{lem:quad}
Let $\sigma(J)\subset[a,b]$, $\kappa=b/a$,
$f(t)=\log\bigl(\tfrac{b-a}{2}t+\tfrac{b+a}{2}\bigr)$ and
$1<\rho<\rho^\star=(\sqrt\kappa+1)/(\sqrt\kappa-1)$, so that $f$ is
analytic inside the Bernstein ellipse $E_\rho$. Let
$M_\rho=\max_{z\in E_\rho}|f(z)|$. For every unit vector $g$, $m$ steps
of Lanczos started at $g$ satisfy
\[
\bigl|g^TEg-e_1^T\log(T_m)e_1\bigr|\;\le\;\eta_m:=
\frac{4M_\rho}{(\rho-1)\rho^{2m-1}} .
\]
\end{lemma}

\begin{proof}
The $m$-point Gauss rule is exact for polynomials of degree $2m-1$, and
both the spectral measure and the Gauss rule are nonnegative with unit
mass. The error is therefore at most
$$2\min_{\deg q\le 2m-1}\Vert f-q\Vert_{\infty,[-1,1]}\le
2\cdot 2M_\rho\rho^{-(2m-1)}/(\rho-1)$$ by the Chebyshev truncation
bound~\cite[Thm.~8.2]{trefethen2019approximation}. The exponent $2m-1$ corrects the
$2m$ in~\cite{ubaru2017fast}; see~\cite{cortinovis2022randomized} and
\cite[Lemma~4.1]{han2023suboptimal}. %Since
% $\log\rho^\star\approx 2/\sqrt\kappa$, accuracy $\eta_m\le\epsilon$
% requires $m=\mathcal{O}\bigl(\sqrt\kappa\log(M_\rho/\epsilon)\bigr)$.
% Note that $M_\rho\to\infty$ as $\rho\to\rho^\star$, so $\rho$ must be
% fixed strictly inside the admissible range;
% \cite[Thm.~4.1]{han2023suboptimal} gives one explicit choice of $\rho$
% together with $M_\rho\le|\log(\lambda_{\min}/2)|+\pi$.
\end{proof}

% \begin{corollary}[Stochastic and quadrature error of DSLQ]\label{cor_DSLQ}
% Let $v_1,\dots,v_r$ be independent Rademacher vectors, independent of
% $S$. Set $z_j=(I-P)v_j$ and $u_j=z_j/\Vert z_j\Vert_2$, and let
% $\mathcal{I}_m(\cdot)$ denote the $m$-step Gaus-Lanczos
% approximation of $g\mapsto g^TEg$. Define
% $$
% \hat t_{\mathrm{DSLQ}}=\sum_{i=1}^{s}\mathcal{I}_m(a_i)
% +\frac1r\sum_{j=1}^{r}\Vert z_j\Vert_2^2\,\mathcal{I}_{m'}(u_j).
% $$
% Then, conditionally on $P$,
% \begin{equation}\label{eq:cond}
% \mathbb{E}_v\bigl[\vert\mathrm{tr}(E)-\hat t_{\mathrm{DSLQ}}|\big\vert P\bigr]
% \le \sqrt{\tfrac{2}{r}}\;\Vert R_P\Vert_F+s\,\eta_m+n\,\eta_{m'} .
% \end{equation}
% Averaging over a Gaussian sketch,
% \begin{equation}\label{eq:uncond}
% \mathbb{E}_{S,v}|\mathrm{tr}(E)-\hat t_{\mathrm{DSLQ}}|\le
% \sqrt{\tfrac{2}{r}}\Bigl(\Vert L_2\Vert_F^2+\tfrac{\Vert\Lambda_2-\gamma I\Vert_F^2}{p-1}
% \bigl\Vert(\Lambda_1-\gamma I)^{-1}L_1\bigr\Vert_F^2\Bigr)^{1/2}
% +s\,\eta_m+n\,\eta_{m'} .
% \end{equation}
% If in addition the assumption of \Cref{prop_sketch_J} holds
% (in particular, always for $\gamma=1$) and $p=k+1$, then
% \begin{equation}\label{eq:rate}
% \mathbb{E}_{S,v}|\mathrm{tr}(E)-\hat t_{\mathrm{DSLQ}}|\le
% \frac{\mathrm{tr}(E)}{\sqrt{r\,k}}+(2k+1)\,\eta_m+n\,\eta_{m'} .
% \end{equation}
% \end{corollary}
Let $v_1,\dots,v_r$ be independent Rademacher vectors, independent of $S$.
Set $z_j=(I-P)v_j$ and $u_j=z_j/\Vert z_j\Vert_2$, and let
$\mathcal{I}_m(\cdot)$ denote the $m$-step Gauss-Lanczos approximation of
$g\mapsto g^\top Eg$. The DSLQ estimator is
$$
  \hat t_{\mathrm{DSLQ}}
  =\sum_{i=1}^{s}\mathcal{I}_m(a_i)
  +\frac1r\sum_{j=1}^{r}\Vert z_j\Vert_2^2\,\mathcal{I}_{m'}(u_j).
$$
We write $\delta_{\mathrm{quad}}=s\,\eta_m+n\,\eta_{m'}$ for the total
quadrature error.

\begin{corollary}[Stochastic and quadrature error of DSLQ]\label{cor_DSLQ}
Conditionally on $P$,
\begin{equation}\label{eq:cond}
  \mathbb{E}_v\bigl[\,|\mathrm{tr}(E)-\hat t_{\mathrm{DSLQ}}|\;\big|\;P\bigr]
  \le\sqrt{\tfrac{2}{r}}\,\Vert R_P\Vert_F+\delta_{\mathrm{quad}} .
\end{equation}
For a Gaussian sketch,
\begin{equation}\label{eq:uncond}
  \mathbb{E}_{S,v}|\mathrm{tr}(E)-\hat t_{\mathrm{DSLQ}}|
  \le\sqrt{\tfrac{2}{r}}
  \Bigl(\Vert L_2\Vert_F^2+\tfrac{\Vert\Lambda_2-\gamma I\Vert_F^2}{p-1}
  \bigl\Vert(\Lambda_1-\gamma I)^{-1}L_1\bigr\Vert_F^2\Bigr)^{1/2}
  +\delta_{\mathrm{quad}} .
\end{equation}
If moreover Proposition~\ref{prop_sketch_J} holds (always for $\gamma=1$) and
$p=k+1$, then
\begin{equation}\label{eq:rate}
  \mathbb{E}_{S,v}|\mathrm{tr}(E)-\hat t_{\mathrm{DSLQ}}|
  \le\frac{\mathrm{tr}(E)}{\sqrt{r\,k}}+(2k+1)\,\eta_m+n\,\eta_{m'} .
\end{equation}
\end{corollary}

\begin{proof}
For Rademacher $v$,
$$\mathrm{Var}(v^TR_Pv)=2\sum_{i\neq j}(R_P)_{ij}^2\le2\Vert R_P\Vert_F^2.$$
Hence the $r$-sample average has variance at most $$2\Vert R_P\Vert_F^2/r,$$ and
$$\mathbb{E}|X-\mathbb{E} X|\le\sqrt{\operatorname{Var}X},$$ gives the first term
of~\eqref{eq:cond}. The quadrature terms follow from
Lemma~\ref{lem:quad}, using $\Vert a_i\Vert_2=1$ and $\Vert z_j\Vert_2^2\le n$.
For~\eqref{eq:uncond}, apply Jensen's inequality,
$$\mathbb{E}_S\Vert R_P\Vert_F\le(\mathbb{E}_S\Vert R_P\Vert_F^2)^{1/2},$$ then Lemma~\ref{lem_range_induced}
and Theorem~\ref{thm:gauss}. For~\eqref{eq:rate}, use~\eqref{eq:hmt}
with $$1+k/(p-1)=2,$$ and $$\Vert L_2\Vert_F\le\mathrm{tr}(E)/(2\sqrt k),$$ which holds
because $E\succeq0$~\cite[Lemma~7]{meyer2021hutch++}.
\end{proof}

\begin{remark}
The current Corollary~3.4 bounds $\Vert R_P\Vert_F$ by $\mathrm{tr}(R_P)$. For the
nearly flat log-spectra of the test matrices, $\mathrm{tr}(R_P)\approx\mathrm{tr}(E)$,
so that bound only guarantees a relative error of order
$\sqrt{2/r}=\mathcal{O}(1)$. By contrast,~\eqref{eq:rate} gives a
relative stochastic error of at most $1/\sqrt{rk}$. With
$r\asymp k\asymp m_{\mathrm{vec}}/3$, this is
$\mathcal{O}(1/m_{\mathrm{vec}})$: the Hutch++ rate in the number of quadratic forms. 
\end{remark}

\begin{remark}[The scaling $\alpha$ does not affect the subspace when $\gamma=0$] 
For $\gamma=0$, $\mathrm{range}(JS)=\mathrm{range}(QS)$, so $P$ is independent of
$\alpha$. Gauss quadrature is exact on constants and $T_m(J)=T_m(Q)/\alpha$,
so $\mathcal{I}_m(g;\log (J))=\mathcal{I}_m(g;\log (Q))-\log\alpha$. Hence
\[
n\log\alpha+\hat t_{\mathrm{DSLQ}}(J)
=\hat t_{\mathrm{DSLQ}}(Q)
+\log\alpha\,\Bigl(n-s-\frac1r\sum_{j=1}^{r}\Vert z_j\Vert_2^2\Bigr),
\]
where the term in  bracket has zero mean. The scaling therefore only adds a
zero-mean term whose variance grows with $\log^2\alpha$. Lemma~\ref{lem_range_induced}
and Theorem~\ref{thm:gauss} do not use $E\succeq0$; positivity enters
only through $\Vert L_2\Vert_F\le\mathrm{tr}(E)/(2\sqrt k)$ in~\eqref{eq:rate}.
\end{remark}

\section{Numerical Experiments}
\label{Kry: experiments}
\label{experiments}
In this section, we evaluate the proposed DSLQ method from two complementary perspectives. Firstly, we investigate the quality of the randomized approximation space and numerically assess the residual bound derived in \Cref{thm:gauss}. Secondly, we compare  the proposed log-determinant method against a wide variety of sparse matrices, and evaluate the   application of the method  to Maximum likelihood estimation and to Bayesian inverse problem for  computing  expected information gain.  We compare the accuracy and computational efficiency of our method  with the stochastic Lanczos Quadrature method (SLQ) \cite{ubaru2017fast}, the Suboptimal subspace construction for log-determinant approximation (So-SLQ)\cite{han2023suboptimal}, funNyström++ \cite{persson2023randomized} and Krylov-aware \cite{chen2023krylov} under a common matrix - vector products budget. For DSLQ, the approximation space
is constructed from $ Y=(J-\gamma I)S$, with $\gamma=1$ and the budget was distributed among the randomized range construction, the projected and residual Lanczos degrees were adjusted  with the available budget. For the remaining methods, their corresponding sketch sizes, number of probes, and krylov depths were adjusted according to their respective matvec cost models.   All experiments were performed on a \textbf{MacBook Pro with Apple M1 Pro chip and 16 GB RAM, using Python 3.11.8}. 
Codes to reproduce the log-det  results in this paper are available at \\ \url{https://github.com/VerlonRoelMBINGUI/Optimal-Krylov-for-log-det-est.git}.

\subsection{Randomized of range comparison and validation of the residual bound}
In this section, we first examine the effect of the randomized range construction on the
approximation of
$$
    E=\log(J),
$$
and use the same experiment to assess the expected residual bound of \Cref{thm:gauss}. We compare the three spaces
$$
    \operatorname{range}(JS),\qquad
    \operatorname{range}((J-I)S),\qquad
    \operatorname{range}(\log(J)S),
$$
where $S\in\mathbb{R}^{n\times(k+p)}$ has independent standard
Gaussian entries. The first construction uses the original scaled matrix,
the second corresponds to the shifted DSLQ construction with
$\gamma=1$, and the third is included as a reference since it constructs
the randomized space directly from the target matrix function.

Synthetic SPD matrices are generated as
$$
    J=U\Lambda U^T,
$$
where $U$ is orthogonal and $\Lambda$ contains prescribed eigenvalues.
We consider several spectral profiles, including exponential decay,
algebraic decay, an approximately flat spectrum, a spectrum with its tail
clustered near one, and an identity-plus-low-rank spectrum. These cases
are intended to represent different levels of spectral decay and
low-rank structure.

Unless otherwise stated, we use
$$
    n=400,\qquad k=10,\qquad p=12,
$$
so that the sketch dimension is \(s=k+p=22\). For every Monte Carlo
realization, the same Gaussian sketch \(S\) is used for all three
constructions. The corresponding orthonormal bases are
$$
    A_J=\operatorname{orth}(JS),\qquad
    A_{\rm shifted}=\operatorname{orth}((J-I)S),\qquad
    A_{\log}=\operatorname{orth}(\log(J)S),
$$
with projectors
$$
    P_J=A_JA_J^T,\qquad
    P_{\rm shifted}=A_{\rm shifted}A_{\rm shifted}^T,\qquad
    P_{\log}=A_{\log}A_{\log}^T.
$$

The approximation quality is measured by
$$
    \mathcal{R}(P)
    =
    \|(I-P)\log(J)\|_F^2.
$$
A smaller value of $\mathcal{R}(P)$ indicates that the corresponding
randomized range captures more effectively the dominant components of
$\log(J)$.

For each spectral configuration, we perform
$N_{\rm trial}=400$ independent Monte Carlo realizations and report
$$
    \widehat{\mathbb{E}}[\mathcal{R}(P)]
    =
    \frac{1}{N_{\rm trial}}
    \sum_{\ell=1}^{N_{\rm trial}}
    \mathcal{R}\!\left(P^{(\ell)}\right).
$$
The empirical residuals are compared with the theoretical upper bounds
corresponding to the three range constructions. In particular, for the
shifted space $\operatorname{range}((J-\gamma I)S)$, we compare the
Monte Carlo estimate with the bound established in
Theorem~\ref{thm:gauss}. We also report the optimal
rank-$k$ spectral tail $ \|L_2\|_F^2 $
as a reference.\\
\quad\\
The results reported in \Cref{fig_randomized_range_comparison} and \Cref{fig_Log_theorem_2_2_validation} show that the shifted construction $Y=(J-I)S$ is generally at least as effective as unshifted space $JS$, while its advantage becomes particularly pronounced when the spectrum is clustered near one or exhibits an identity plus-low-rank structure. In these cases, subtracting the identity suppresses eigen-directions associated with $\lambda_i \approx 1$, whose contribution to $\log(J)$ is small, and therefore yields a randomized space that is more closely aligned with the dominant components of the matrix logarithm. Consequently, the residual $\vert (I-P)\log(J)\vert_F^2$ obtained from $(J-I)S$ can be several orders of magnitude smaller than that obtained  form $JS$, and in the identity-plus-low-rank regime it is comparable top the residual generated directly from $\log(J)S$. Moreover, for all tested spectra profiles, the empirical results associated with the shifted construction remains below theoretical upper bound of \Cref{thm:gauss}, providing numerical support for the expected residual analysis. The bound is relatively sharp for several spectra and becomes more conservative when the trailing eigenvalues are strongly clustered near one. 
\begin{figure}[H]
\centering
\includegraphics[width=1\textwidth]{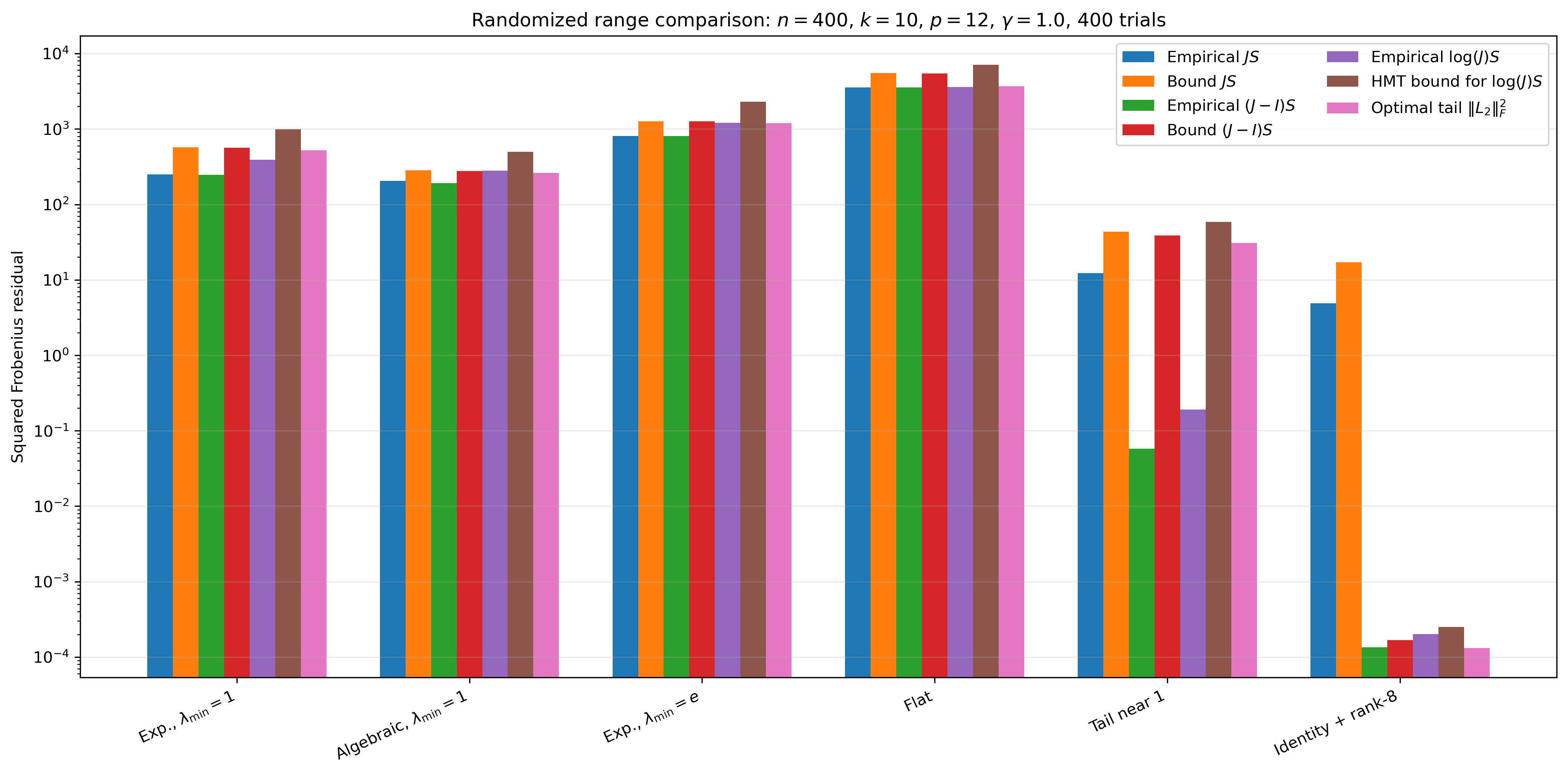}
\caption{Randomized range comparison : : Monte Carlo average residual error over 400 independent realizations}
\label{fig_randomized_range_comparison} 
\end{figure}

\begin{figure}[H]
\centering
\includegraphics[width=1\textwidth]{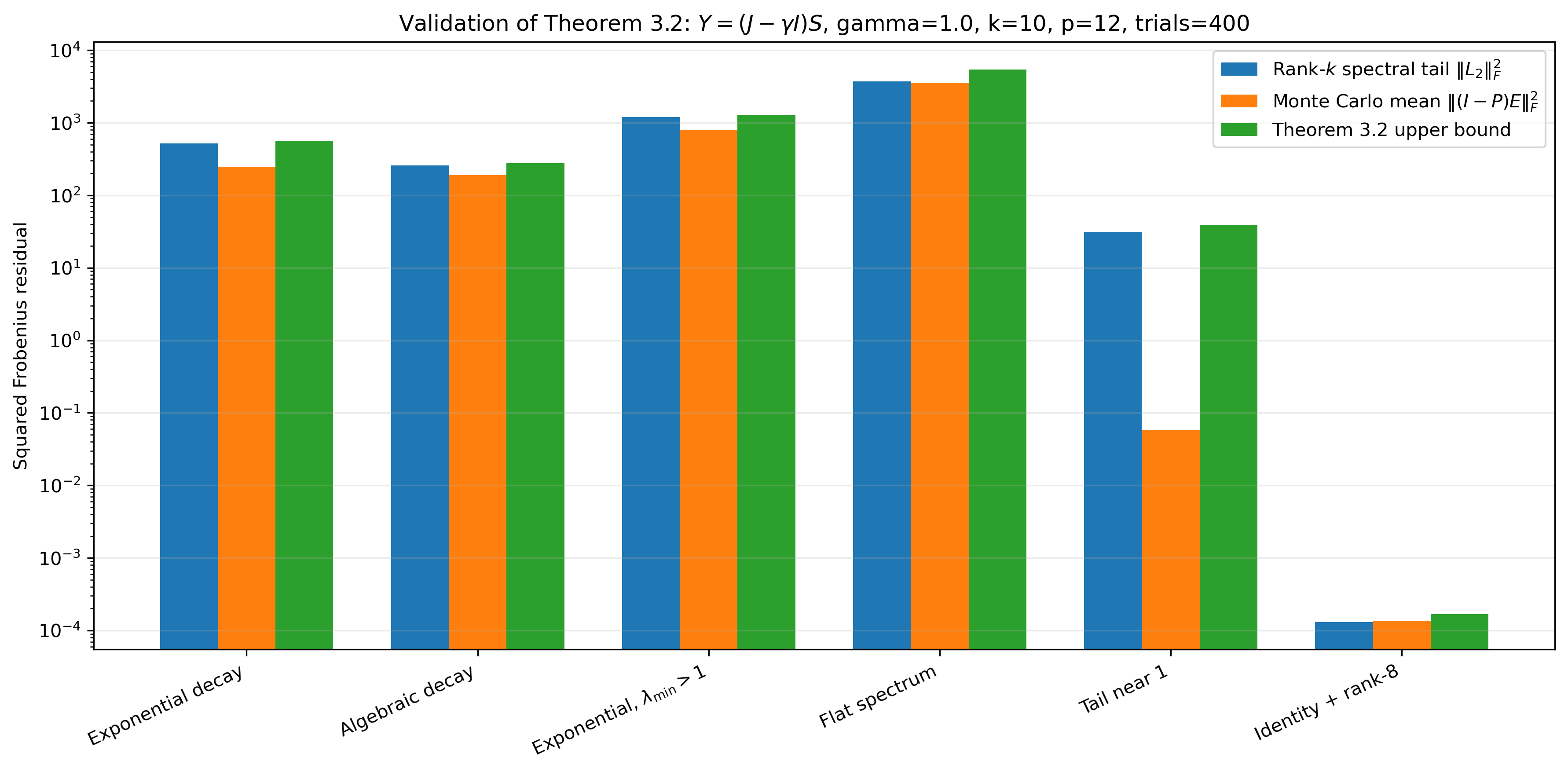}
\caption{ Residual bound validation (\Cref{thm:gauss}): empirical residual error versus the theoretical bound over 400 Monte Carlo realizations}
\label{fig_Log_theorem_2_2_validation} 
\end{figure}
\subsection{Log-determinant estimation}
In this section, our benchmark includes real data sets from the University of Florida Sparse Matrix Collection \cite{davis2011university}  carefully selected to provide controlled tests. The actual matrices occur over a range of important applications areas arise in computational fluid dynamics simulations, from finite element models in structural engineering, elasticity benchmarks used  material science, plasma physics, and modeling crystalline structures, etc. (see  \Cref{table 1 : kry}).\\
\quad\\
This broad spectrum of matrices allows us to experiment with our method under all situations: very sparse and unstructured meshes (finite element analysis, computational fluid dynamics), structured systems based on graphs (power grids), dense local coupling (molecular and crystal models), physics-based operators (gyrokinetic plasma simulation). Such variation is required because it evidences how  our  methods : DSLQ  performs in practice under differing sparsity patterns and spectral distributions (see Table \ref{table 1 : kry}).
\begin{table}[H]
\caption{Description of sparse matrices used in the experiments. 
Real matrices are from the UF Sparse Matrix Collection \cite{davis2011university}.}
\label{table 1 : kry}
\centering
\renewcommand{\arraystretch}{1.5}
\begin{tabular}{|l| l| c|}
\hline
Matrix        & Application Domain              & Size ($n \times n$) \\
\hline
\texttt{pdb1HYS}      &  protein structure (bioinformatics)  & $36,417 \times 36,417$ \\
\texttt{boneS01}      & Model Reduction Problem   & $ 127,224 \times 127,224$ \\
\texttt{cfd1}  &   Fluid dynamics          & $70,656 \times 70,656$ \\
\texttt{ecology2} &   Circuit theory      & $ 999,999
 \times 999,999
$ \\
\texttt{af\_shell8}  &  Structural mechanics and Sheet metal forming      & $ 504,855 \times 504,855 $ \\
\texttt{bone010}  & Biomechanics    & $986,703 \times 986,703 $ \\
\texttt{thermomech\_TC}  & Thermal problem & $ 102,158 \times 102,158 $ \\
\texttt{audikw\_1}   & Structural Mechanics & $943,695 \times 943,695$ \\
 \texttt{gyro\_k} & Plasma physics (fusion)   & $ 17,361 \times 17,361 $ \\
 \texttt{wathen120}  & 2D FEM elasticity benchmark  & $36,441 \times 36,441$ \\
\hline
\end{tabular}
\end{table}
\paragraph{\textbf{Comparison under Equal Matvec Budget.}}
We evaluate the competing log-determinant estimators under a common matrix-vector product budget.  Across both matrices \texttt{pdb1HYS} and  \texttt{wathen120}, DSLQ reaches the relative error of $10^{-3}$ with comparatively low computational effort. It achieves   the threshold at modest matrix–vector budgets and short CPU times, while remaining competitive with or faster than the other methods. In particular, DSLQ shows smooth and stable convergence, whereas funNyström++ and Krylov-aware exhibit stronger non-monotonic fluctuations.
\begin{figure}[H]
\centering
\includegraphics[width=0.49\textwidth]{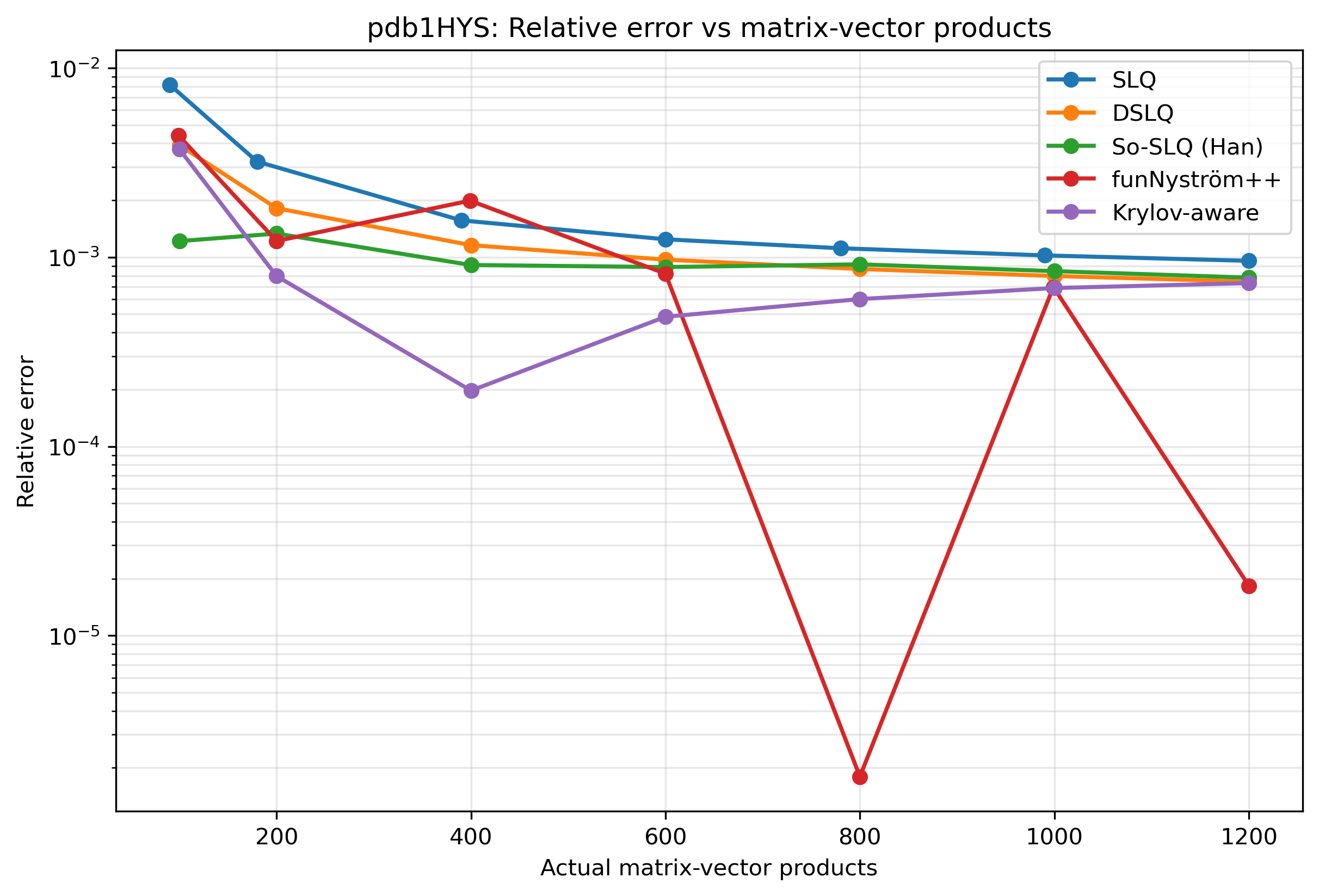}
\includegraphics[width=0.49\textwidth]{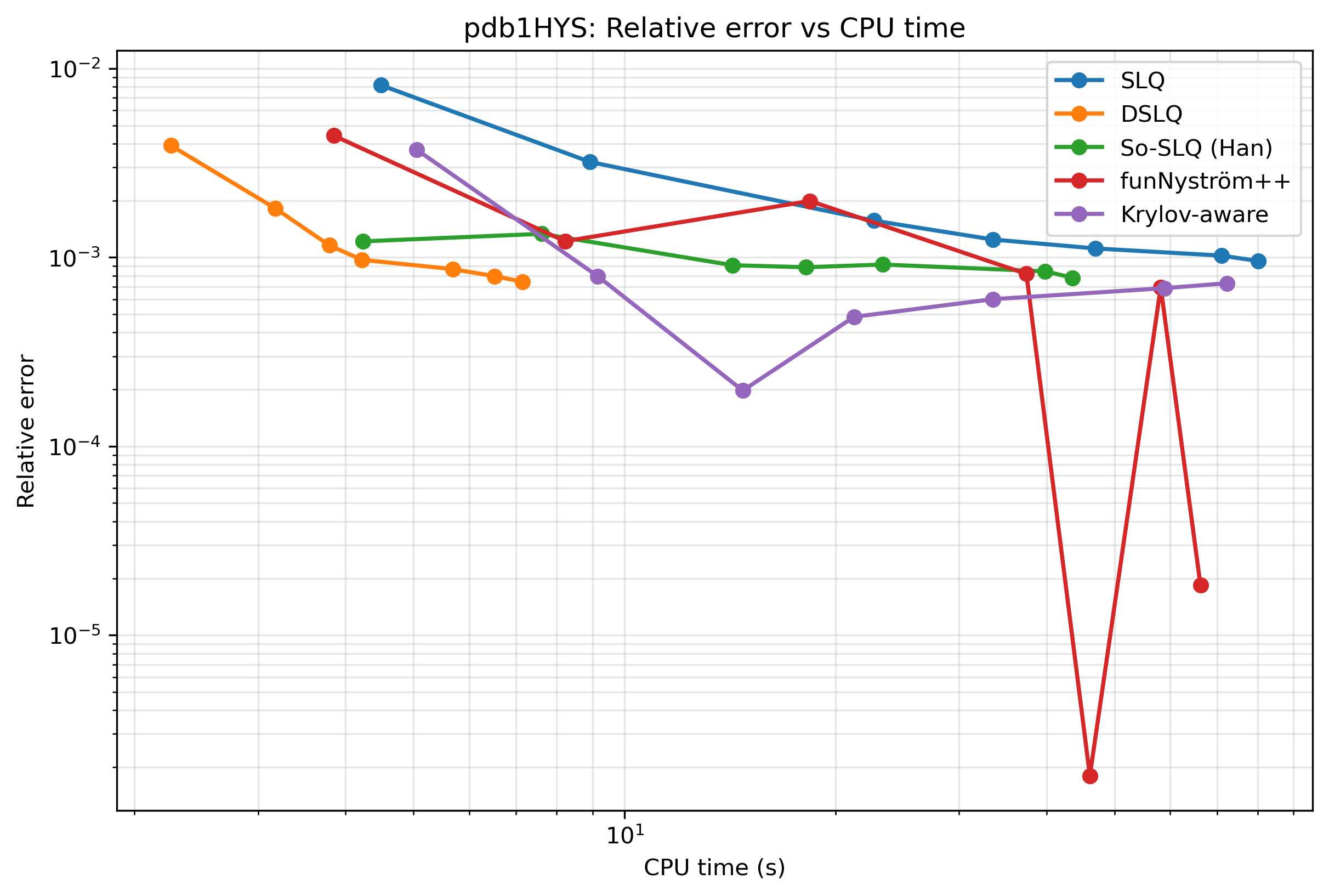}\\
\includegraphics[width=0.49\textwidth]{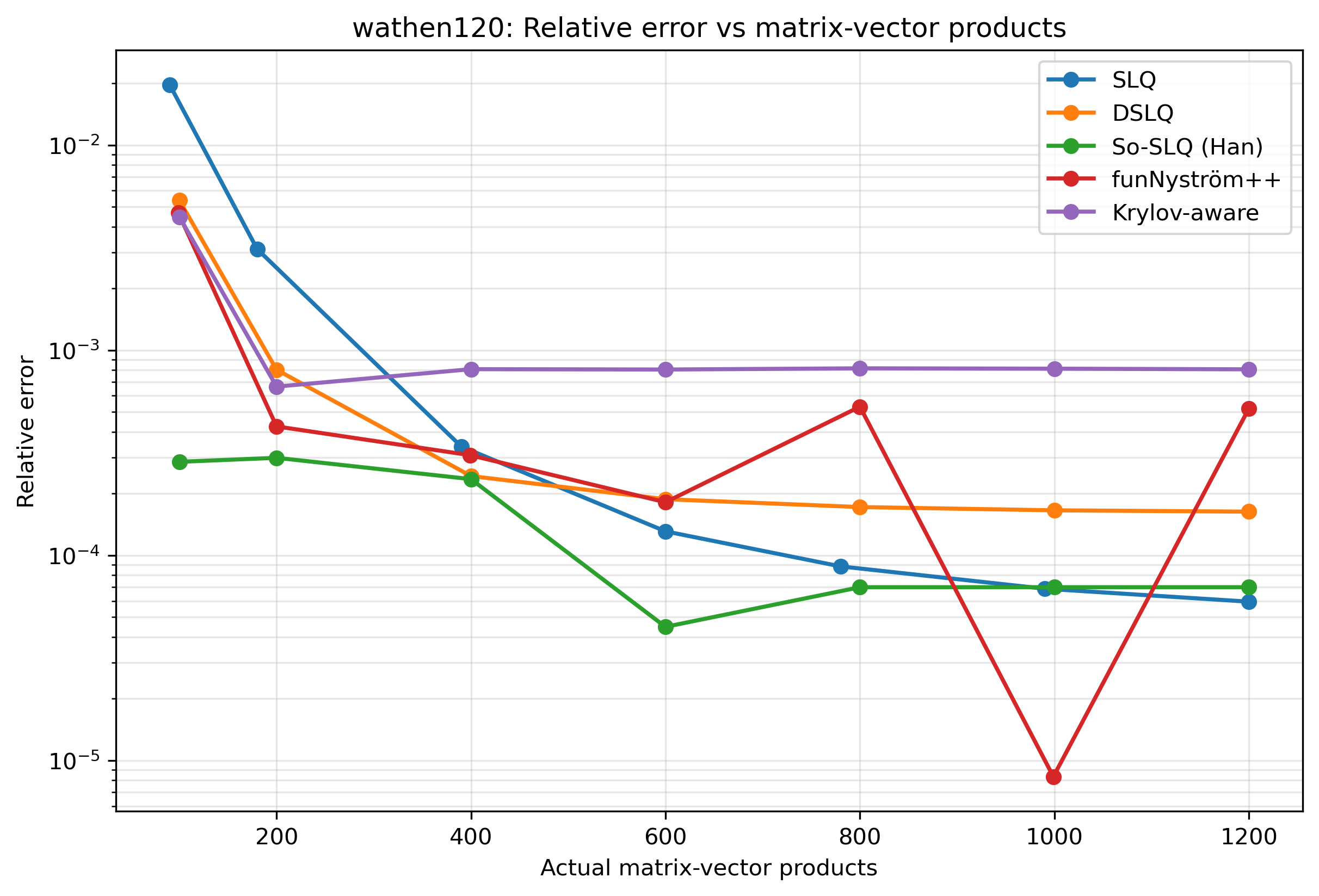}
\includegraphics[width=0.49\textwidth]{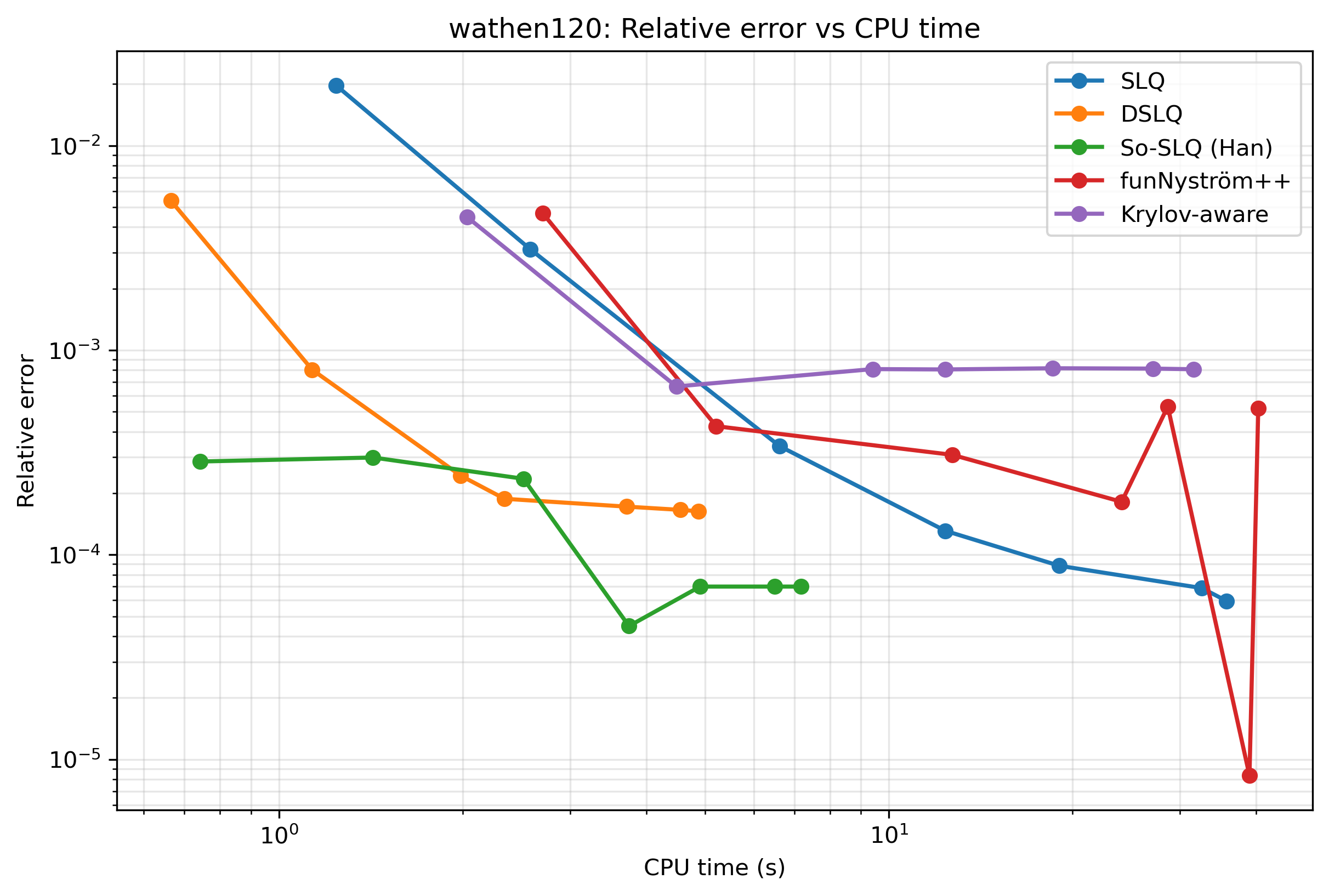}
\caption{Method comparison under a common matrix-vector budget over 20 trials. The  first row is for \texttt{pdb1HYS}  and the second for \texttt{wathen120}}
\label{fig_Log_comparison_DSLQ} 
\end{figure}
\clearpage
\paragraph{\textbf{UFL sparse matrice Collections with a budget of $B=1800$.}}
In Table~\ref{Kry : table 2}, we present the average log-determinant estimation for a set of benchmark matrices. The average are taken over 10 trials for the large matrices  \texttt{ecology2}, \texttt{boneS01}, \texttt{af\_shell8}, \texttt{bone010}\texttt{audikw\_1}  and 20 trials for the remaining  moderate-size matrices. The exact log-determinants are computed using the Cholesky factorization available in \emph{SciPy}.
\begin{table}[H]
\caption{Comparison of log-determinant estimation between our method (DSLQ) and So-SLQ, SLQ, funNystr\"om++, and Krylov-aware. ``Estimate'' denotes the average computed log-determinant and ``CPU time'' the average runtime in seconds over 10 trials (large matrices) or 20 trials (remaining matrices), using a budget of $B=1800$ matrix-vector products. OOM (out of memory) indicates that the exact log-determinant could not be computed on our hardware.}
\label{Kry : table 2}
\centering
\renewcommand{\arraystretch}{3} 
\setlength{\tabcolsep}{2.5pt}      
\small                          
\resizebox{\textwidth}{!}{%
\begin{tabular}{|l|c|cc|cc|cc|cc|cc|}
\hline
Matrices &  Exact logdet
& \multicolumn{2}{c|}{SLQ} 
& \multicolumn{2}{c|}{So-SLQ}
& \multicolumn{2}{c|}{DSLQ} 
& \multicolumn{2}{c|}{ funNystöm++} 
& \multicolumn{2}{c|}{ krylov\_aware}
 \\
\cline{1-12}
&  Value 
& Estimate & CPU time (s)
 &  Estimate & CPU time (s)
&  Estimate & CPU time (s)
&  Estimate & CPU time (s) 
&  Estimate & CPU time (s)
   \\
\hline
\texttt{ecology2} & OOM  & 3394410.612& 114.490& 3394220.283& 31.442 & 3394073.236&16.488&3395612.152 &582.084 &3394458.822&227.624\\ 
\texttt{boneS01} & OOM  & 1103969.797&  28.062 &1103862.663 & 8.526& 1103910.838 &5.953 &  1103593.911 &27.235&1104001.775&31.045 \\
\texttt{af\_shell8} & OOM  &6466429.965 &82.002 &6466362.670 &28.582 &    6466323.092&20.898 &  6465043.506&101.115 &6466026.864&105.597\\
\texttt{bone010} & OOM & 9290577.518& 164.177 &  9290673.465& 72.826 &9290642.576 &52.63 &9290638.089 &590.109&9290474.864&286.223\\
\texttt{wathen120}& 127080.312 &127086.779 & 9.922&127076.651 & 1.541 & 127110.121 & 1.188 &126970.690 &6.291&127060.142&6.902\\
\texttt{thermomech\_TC}& - 546787 & -546791.405&18.793 &-546805.572& 3.609& -546821.286 & 2.422 &-546769.076&18,934&-546734.962&17.628\\
\texttt{pdb1HYS}& 170081.683 & 170224.097&13.104 & 170181.891& 5.849 & 170182.080& 4.343&169830.144 &14.067 &170149.087&16.197\\
\texttt{cfd1}&  - 42892.386 &  -42870.805 &   14.550& -42905.254 & 3.608&  -42865.332& 2.348& -42957.628&14.168 &-42736.716&14.121\\
\texttt{gyro\_k} & 304741.205   & 305007.696 & 4.950& 304852.718&2.437 & 304931.883&  1.066&304704.352&4.048 &304731.649&4.777\\
\texttt{audikw\_1}& OOM & 13304939.294& 221.023&  13255549.970& 109.119 & 13288171.574&79.669 &13251367.835 &423.722 &13250153.236&369.367\\
\hline
\end{tabular} % 
}
\end{table}

\begin{table}[H]
\caption{Average Relative error comparison between SLQ, So-SLQ, DSLQ,  funNystöm++ and Krylov\_aware over 20 trials. The exact Log-det is computed using Cholesky factorization.}
\label{Kry : table 3}
\centering
\renewcommand{\arraystretch}{1.5} 
\setlength{\tabcolsep}{1.5pt}      
\small                          
\resizebox{\textwidth}{!}{%
\begin{tabular}{|l|c|c|c|c|c|}
\hline
Matrices        & SLQ Rel. error & So-SLQ Rel. error & DSLQ Rel. error             &  funNystöm++ Rel. error & Krylov\_aware  \\
\hline

\texttt{pdb1HYS}      & $ 8.373 \times 10^{-4}$   & $5.892 \times 10^{-4}$ & $5.903 \times 10^{-4} $& $1.479 \times 10^{-3}$ & $3.963 \times 10^{-4}$\\
\texttt{cfd1}  &  $ 5.031 \times 10^{-4}$  & $ 3.000 \times 10^{-4}$ & $6.307 \times 10^{-4}$ & $1.521 \times 10^{-3}$& $3.629 \times 10^{-3}$\\
\texttt{thermomech\_TC}  & $8.056 \times 10^{-6}$ & $3.397 \times 10^{-5}$& $6.270 \times 10^{-5}$ & $3.278\times 10^{-5}$ &$9.517 \times 10^{-5}$\\
 \texttt{gyro\_k} & $8.745 \times 10^{-4}$ & $3.659 \times 10^{-4}$& $6.257 \times 10 ^{-4}$& $ 1.209\times 10^{-4}  $ &$ 3.136\times 10^{-5}  $\\
 \texttt{wathen120}  & $5.089 \times 10^{-5}$  &$ 2.881 \times 10^{-5}$ & $2.346 \times 10^{-4}$ & $ 8.626\times 10^{-4}$ &$ 1.587\times 10^{-4}$\\
\hline
\end{tabular}% 
}
\end{table}

\begin{table}[H]
\caption{Average relative error comparison between SLQ, So-SLQ, DSLQ  funNystöm++ and Krylov\_aware over 20 trials. The the average reference log-det is computed using SLQ with a high budget of $B=9000$ over 5 trials.}
\label{Kry : table 4}
\centering
\renewcommand{\arraystretch}{1.2} 
\setlength{\tabcolsep}{1.2pt}      
\small                          
\resizebox{\textwidth}{!}{%
\begin{tabular}{|l|c|c|c|c|c|}
\hline
Matrices  &SLQ Rel. error     &  So-SLQ Rel. error & DSLQ Relative error             &   funNystöm++ Rel. erro& Krylov\_aware \\
\hline

\texttt{ecology2}      & $8.145 \times 10^{-5}$ &$1.375\times 10^{-4}$  &$1.808 \times 10^{-4}$ &  $2.725\times 10^{-4}$&$6.725\times 10^{-5}$\\
\texttt{boneS01} & $3.322\times 10^{-5}$ &  $1.303 \times 10^{-4}$ & $8.662\times 10^{-5}$ & $3.737\times 10^{-4}$&$4.253\times 10^{-6}$\\
 \texttt{af\_shell8}& $6.024\times 10^{-5}$ &$1.041\times 10^{-5} $ &$1.653\times 10^{-5}$ & $ 2.144\times 10^{-4}$  &$6.234\times 10^{-5}$  \\
 \texttt{bone010}  &$4.000\times 10^{-6}$  & $6.327\times 10^{-6}$&$3.002\times 10^{-6}$  &$2.519\times 10^{-6}$ &$1.505\times 10^{-5}$  \\
  \texttt{audikw\_1} &$4.091\times 10^{-3}$ &$3.639 \times 10^{-4}$ & $2.826\times 10^{-3}$ & $4.824\times 10^{-5}$ &$4.342\times 10^{-5}$\\
\hline
\end{tabular}% 
}
\end{table}
In \Cref{Kry : table 3}, we report the relative errors obtained by the different methods SLQ, So-SLQ, DSLQ,  funNystöm++ and Krylov\_aware on the medium-sized matrices with a budget of number of matvec of  $B=1800$, using the Cholesky factorization as the reference (exact) log-determinant. DSLQ achieves a relative error within the interval  $[10^{-5},10^{-4}]$ across all tested matrices, remaining competitive with SLQ and So-SLQ,funNystöm++ and Krylov\_aware. In \Cref{Kry : table 4}, across the tested large-scale matrices, DSLQ achieves the relative error within the interval $[10^{-6},10^{-3}]$. These results indicates that no single method provides the smallest error across all matrices.\\
Taken together, \Cref{Kry : table 3}, \Cref{Kry : table 4} and \Cref{Kry : table 2} show that DSLQ provides a favorable accuracy-cost trade-off. While SLQ and So-SLQ, Krylov\_aware  can yield smaller errors on some matrices. DSLQ remains of comparable accuracy over most of the benchmark set and is consistently the fastest runtime in the reported runtime experiment (cf. \Cref{Kry : table 2}).
\FloatBarrier
\subsection{Maximum likelihood estimation for GRFs} We now evaluate our proposed approach for maximum likelihood estimation of Gaussian Random Fields (GRFs) (DSLQ). We simulate on such a field using the Wendland $C^2$ covariance function \cite{williams2006gaussian} (with smoothness $q=1$) on a $900 \times 1200 $ grid ($n=1.08\times 10^6)$ to demonstrate the usage of Log-determinant computations in GRFs, as in \cite{ubaru2017fast}. The goal is to compute the grid maximizer 
\begin{equation}
    \widehat \theta = \arg \max_{\theta \in \Theta}\mathcal{L}(\theta), \quad \mathcal{L}(\theta) = \dfrac{1}{2}\log\det(Q(\theta))- \dfrac{1}{2}x^\top Q(\theta)x-\dfrac{n}{2}\log(2\pi).
\end{equation}
Using a matrix free - evaluation of $\log\det(Q(\theta))$ at each candidate $\theta \in \Theta=\{20,30,40,\\50,60,70,80\}$.
Set $\mathrm{tr}\log(Q(\theta))\approx \widehat{tr}(\log(Q(\theta))$.\\

Using DSLQ,  the $\widehat{ \mathrm{tr}}$ is computed using  \Cref{algo_log_det_DSLQ}. Because \Cref{algo_log_det_DSLQ} works on PSD matrices, to ensure that, we work on a shifted operator
$$J=Q(\theta)/\alpha\qquad \text{with} \quad \alpha = \lambda_{min}(Q(\theta)).$$
This guaranties that matrix $J$ is PSD. We then  apply \Cref{algo_log_det_DSLQ} and reconstruct 
$$ \widehat{\log\det(Q(\theta))} = n\log(\alpha) + \widehat{\mathrm{tr}}(\log(J(\theta))).$$
\begin{table}[H]
\caption{Estimated $\theta$ and resulting
log likelihood on the $900\times1200$ grid using DSLQ.}
\label{table_log_likelihood_DSLQ}
\centering
\begin{tabular}{|c|c|c|}
\hline
$\theta$      &  $\widehat{\log \det}Q(\theta)$ &  $L(\theta)$ \\
\hline
20 & $5.5317\times 10^{6}$  & $1.5571\times 10^{6}$\\
30 & $6.2357\times 10^{6}$ & $1.7713\times 10^{6}$\\
40 & $6.5848\times 10^{6}$ &  $1.8437\times 10^{6}$\\
\textbf{50} & $\mathbf{6.7881\times 10^{6}}$ &  $\mathbf{1.8606\times 10^{6}}$\\
60 & $6.9200\times 10^{6}$  & $1.8469\times 10^{6}$\\
70 & $7.0106\times 10^{6}$  & $1.8134\times 10^{6}$\\
80 & $7.0776\times 10^{6}$  & $1.7657\times 10^{6}$\\
\hline
\end{tabular}
\end{table}
\begin{figure}[H]
\centering
\includegraphics[width=0.6\textwidth]{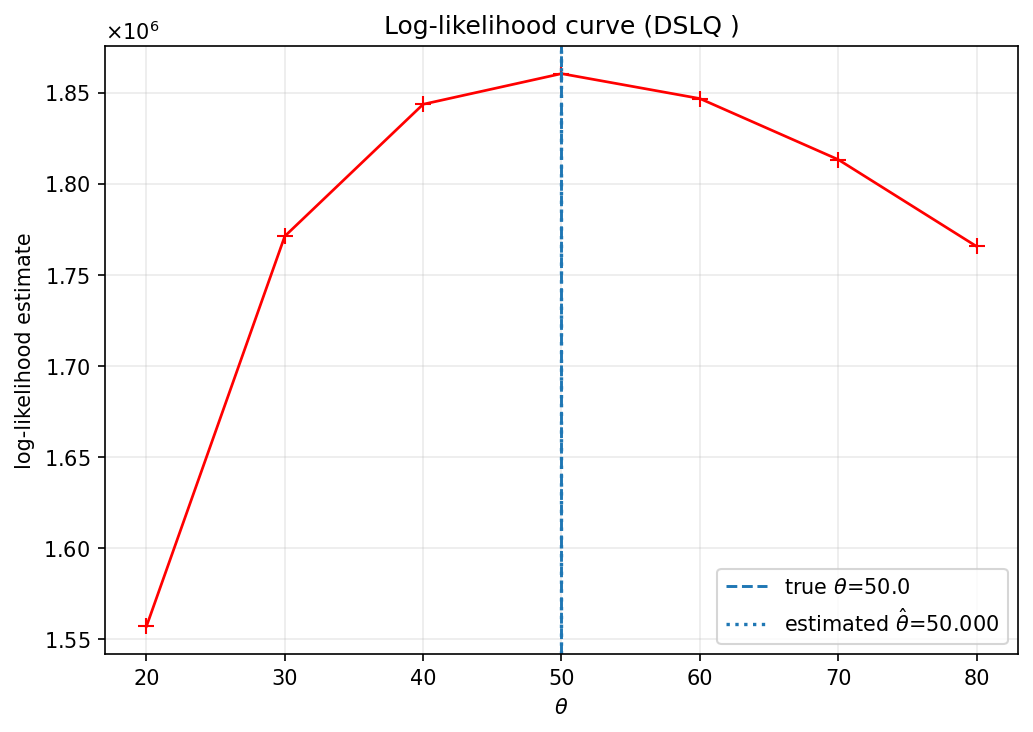}
\caption{Log-likelihood estimation using DSLQ}
\label{fig_Log_likelihood_DSLQ} 
\end{figure}
\Cref{table_log_likelihood_DSLQ} and \Cref{fig_Log_likelihood_DSLQ} show that the peak of the likelihood curve using DSLQ is currently located at $\widehat\theta=50$. Log-determinant were computed using $m=60$ lanczos step, and $m_{vec}=30$ probe vectors.
\FloatBarrier
\subsection{Bayesian inverse problem.}
Inspired by the numerical experiments reported in \cite{persson2023randomized, alexanderian2014optimal,dudley2021monte,saibaba2017randomized}, 
we investigate the parabolic partial differential equation (PDE) governing adavection-diffusion-reaction dynamics over the two- dimensional  spatial domain $D=[0,1]^2$ over time $t\in [0,t_{max}]$:
\begin{equation}
    \begin{cases}
        \dfrac{\partial u}{\partial t}=\kappa \Delta u + \lambda u \qquad \text{in}~ D \times [0,t_{max}]\\
        u(\cdot ,0 ) = \nu \qquad \qquad ~\in D\\
        u=0 \qquad\qquad \qquad ~~\text{on}~ \Gamma_1\\
        \dfrac{\partial u}{\partial n}= \qquad \qquad \qquad \text{on}~\Gamma_2
    \end{cases}
\end{equation}
where $\kappa>0$ is the diffusion coefficient, $\lambda >0$ is the reaction parameter, $\Gamma_2 = \{(x,1)\in \mathbb{R}^2:x\in[0,1]\}$ imposes Neumann boundary conditions, and $\Gamma_1=\partial D \backslash\Gamma_2$ imposes Dirichlet boundary conditions.
\paragraph{Semi-Discret Form. }
Discretizing the spatial domain $D$ unsing finite differences on a $n_x\times n_x$ equispaced grid (N=$n_x^2$ total spatial nodes) yields the system of linear ordinary differential equations :
$$\dot u (t)=Lu(t), \qquad \text{for}\quad t\in [0,t_{max}], \quad u(0)=\nu,$$
where $L\in \mathbb{R}^{N\times N}$ be the sparse spatial discrete operator incorporating boundary conditions. The exacte solution is given by 
$$u(t)=\exp{tL}\nu.$$
\paragraph{Forward observation operator H.}
Place $n_s$ sensors in $[0,1]^2$ and recor measurment at $n_t$ discrete observation time step $\{t_1,\dots,t_{n_t}\}.$ Let $K\in \mathbb{R}^{n_t\times N}$ be the sparse spatial selection matrix. The full observation vector $\mu_{\text{measure}}\in \mathbb{R}^{d}(d=n_s\cdot n_t)$ is mapped linearly from the initial state parameter $\nu \in \mathbb{R}^N$:
$$
u_{\text{measure}} = H\nu, \quad \text{where}\quad H=  \begin{bmatrix}
            K\exp{t_1L}\\
            K\exp{t_2L}\\
            \vdots \\
            K\exp{t_{n_t}L}
        \end{bmatrix} \in \mathbb{R}^{d\times N}.$$

\paragraph{Bayesian Linear Inverse Problem.}
Assuming a Gaussian prior $\eta \sim \mathcal{N}(\eta_0, \Sigma_0)$ and additive Gaussian observational noise $\epsilon\sim \mathcal{N}(0, \Sigma_{\text{noise}}) $, the noisy measurement vector $y= H\eta + \epsilon \in \mathbb{R}^d.$ The resulting posterior distribution is $\eta/y \sim\mathcal{N}(\eta_{\text{post}}, \Sigma_{\text{post}}) $ with the covariance matrix 
$$\Sigma_{\text{post}}=(H^\top\Sigma_{\text{noise}}^{-1} H+\Sigma_0^{-1})^{-1}.$$
The symmetric positive semi-definite (SPSD) matrix $Q_{\text{pde}}\in \mathbb{R}^{N\times N}$ is defined by non dimensionally transforming the prior-weighted forward map : 
$$Q_{\text{pde}}= \Sigma^{1/2}H^\top\Sigma_{\text{noise}}^{-1} H\Sigma^{1/2}.$$
The objective quantity $\mathrm{tr(\log(Q_{\text{pde}} +I)}=\log\det(Q_{\text{pde}}+I)$ measures the expected information gain from the posterior distribution relative to the prior distribution \cite{persson2023randomized,alexanderian2014optimal}. The following table reported the the results from the computation of $\log\det(Q_{\text{pde}}+I)$. Note that, in the experiment, for DSLQ, since, $Q_{\text{pde}}$ is SPSD, the scaling in \Cref{scaled_sigma_matrix_slq} is not required in the \Cref{algo_log_det_DSLQ}.

\begin{table}[htbp]
\caption{Parameters.}
\label{table_param_paper2}
\centering
\begin{tabular}{|c|c|}
\hline
 Parameters &  Values \\
\hline
$n_x$  & $71$\\
 $t_{\max}$  & $2$\\
 $\kappa$ & $0.1$\\
 $\lambda$ &  $0.01$\\
  $\sigma_{\text{prior}}$ & $1$\\
$\sigma_{\text{noise}}$  & $0.05$\\
$\Sigma_{\text{prior}}$ &$\sigma_{\text{prior}}C$  \\
$\Sigma_{\text{noise}}$  & $\sigma_{\text{noise}}^2I_d$\\
 $l$(spatial correlation length scale ) & $0.2$\\
 number of sensors per dimension  & $9$\\
  times   & $\{1.0,1.5,2.0\}$\\
\hline
\end{tabular}
\end{table}
\begin{table}[htbp]
The prior correlation matrix $C\in\mathbb{R}^{N\times N}(N=71^2=5041)$ is constructed using an isotropic Gaussian Kernel defined by
\begin{equation}
    C_{ij}=k(r_i,r_j)=\exp \left(\dfrac{\Vert r_i-r_j\Vert^2_2}{2l^2}\right),
\end{equation}
where $r_i,r_j \in [0,1]^2$ denote the coordinate of the grid.
\caption{  Values of $\mathrm{tr}(\log(Q_\text{pde}+I))$ using SLQ, So-SLQ, DSLQ , funNystöm++ and Krylov\_aware. The exact $\log\det$ is computed using Cholesky factorization over 20 trials.}
\label{tab_inverse_pb_paper2}
\centering
\begin{tabular}{|l|c|c|c|c|c|}
\hline
Exact $\log\det$       & SLQ  & So-SLQ  & DSLQ   & funNystöm++ & Krylov\_aware\\
\hline

 $9.351$     & $ 9.628 \pm 0.029$   & $9.431 \pm 0.0084$ & $9.351 \pm 0.000$&  $9.351 \pm 0.000$& $9.67 \pm 0.034$\\

\hline
\end{tabular}% 
\end{table}
By benchmarking stochastic trace estimators for the expected information gain (EIG = $\mathrm{tr}(\log(Q_\text{pde}+I))$ for $N=5041$, DSLQ achieves  substantial improvement over standard SLQ, Krylov\_aware, So-SLQ reducing the relative error from approximately $2.9\%, 0.84\%, 3.4\%$ respectively to
below the reported numerical precision. Hence, DSLQ estimates the total information gained from the observations with very high accuracy.
\FloatBarrier
\section{Conclusion}
\label{sec_conclusions}
In this work, we have presented DSLQ (Range-Deflated Stochastic Lanczos Quadrature) for the estimation of the log-determinant of large symmetric positive definite matrices. DSLQ combines stochastic Lanczos quadrature with a Hutch++-type residual correction. In particular, the randomized approximation space is constructed using only matrix-vector products with the original scaled symmetric positive definite matrix, while the matrix logarithm is accessed implicitly through Gauss-Lanczos quadrature. We derived an error analysis tailored to this construction, accounting for the residual induced by the randomized approximation space, the stochastic residual estimation error, and the Gauss-Lanczos quadrature errors. Numerical experiments on large sparse matrices from real-world applications demonstrate the accuracy, computational efficiency, and scalability of the proposed method.

\section*{Data availability}
Some of  the data used in our experiments were obtained from  the University of Florida(UFL) Sparse Matrix Collection \cite{davis2011university}.
\section*{Declaration of competing interests}
The authors declare that they have no known competing financial interests or personal relationships that could have appeared to influence the work reported in this paper.
\section*{Acknowledgements}
This work was made possible by a grant from Carnegie Corporation of New York (provided through the AIMS Research and Innovation Centre). The statements made and views expressed are solely the responsibility of the authors.

 \bibliographystyle{plain}  
\bibliography{main}   
\end{document}